\documentclass[12pt]{article}%
\usepackage{amsmath}
\usepackage{amsfonts}
\usepackage{amssymb}
\usepackage{graphicx}%
\providecommand{\U}[1]{\protect\rule{.1in}{.1in}}
\newtheorem{theorem}{Theorem}

\newtheorem{definition}[theorem]{Definition}
\newtheorem{example}[theorem]{Example}

\newtheorem{remark}[theorem]{Remark}

\newenvironment{proof}[1][Proof]{\noindent\textbf{#1.} }{\ \rule{0.5em}{0.5em}}
\begin{document}

\title{Discrete semiclassical orthogonal polynomials of class one}
\author{Diego Dominici \thanks{e-mail: dominicd@newpaltz.edu}\\Department of Mathematics \\State University of New York at New Paltz \\1 Hawk Dr.\\New Paltz, NY 12561-2443 \\USA 
\and Francisco Marcell{\'a}n \thanks{e-mail: pacomarc@ing.uc3m.es}\\Departamento de Matem\'aticas \\Universidad Carlos III de Madrid \\Escuela Polit\'ecnica Superior \\Av. Universidad 30 \\28911 Legan\'es\\Spain}
\maketitle

\begin{abstract}
We study the discrete semiclassical orthogonal polynomials of class $s=1$. By
considering all possible solutions of the Pearson equation, we obtain five
canonical families. We also consider limit relations between these and other
families of orthogonal polynomials.

\end{abstract}

\strut Keywords: discrete orthogonal polynomials, Pearson equation, discrete
semiclassical polynomials, Laguerre-Freud equations, Painlev\'{e} equations

\strut MSC-class: 33C47 (Primary), 34M55, 33E17, 42C05 (Secondary)

\section{Introduction}

Discrete orthogonal polynomials with respect to uniform lattices have
attracted the interest of researchers from many points of view
\cite{MR1149380}. A first approach comes from the discretization of
hypergeometric second order linear differential equations and thus the
classical discrete orthogonal polynomials Charlier, Krawtchouk, Meixner and
Hahn appear in a natural way. As a consequence of the symmetrization problem
for the above second order difference equations, you can deduce the (discrete)
measure with respect to such polynomials are orthogonal. This yields the
so-called Pearson equation that the measure satisfies.

In the last twenty years, new families of discrete orthogonal polynomials have
been considered in the literature taking into account the so-called canonical
spectral transformations of the orthogonality measure. In particular, when you
add mass points to the discrete measure (Uvarov transformation) the sequences
of orthogonal polynomials with respect to the new measure have been studied
extensively (see \cite{MR813269}, \cite{MR1353079}, \cite{MR1379116}, among
others). When you multiply the discrete measure by a polynomial (Christoffel
transformation), some results are known \cite{MR1939588}.

From a structural point of view, some effort has been done in order to
translate to the discrete case the well stated theory of semiclassical
orthogonal polynomials (see \cite{MR1270222}). In particular,
characterizations of such polynomials in terms of structure relations of the
first and second kind, as well as discrete holonomic equations (second order
linear difference equations with polynomial coefficients of fixed degree and
where the degree of the polynomial appears as a parameter) have been done
\cite{MR1665164}. Linear spectral perturbations of semiclassical linear
functionals have been studied in the Uvarov case \cite{MR1467146}.

On the other hand, we must point out that the linear canonical spectral
transformations (Christoffel, Uvarov, Geronimus) of classical discrete
orthogonal polynomials yields discrete semiclassical orthogonal polynomials.
But, as a first step, it remains open the problem of classification of
discrete semiclassical linear functional of class one. The symmetric discrete
semiclassical linear functional of class one have been described in
\cite{MR2457103}. Notice that the classification of $D$- semiclassical linear
functional of class one was done in \cite{MR1186737} and for class two in
\cite{MR2914891}.

The aim of this work is to provide a constructive method of $D_{w}%
$-semiclassical orthogonal polynomials based on the Pearson equation that the
corresponding linear functional satisfies. We will focus our attention on the
classification of $D_{1}$-semiclassical linear functionals of class $s=1.$ In
such a way, new families of linear functionals appear. Notice that an
alternative method is based on the Laguerre-Freud equations satisfied by the
coefficients of the three term recurrence relations associated with these
orthogonal polynomials. Their complexity increases with the class of the
linear functional and the solution is cumbersome. Basic references concerning
this approach are \cite{MR1662690} as well as \cite{MR2457103}.

The structure of the manuscript is as follows. Section 2 deals with the basic
definitions and the theoretical background we will need in the sequel. In
Section 3 we describe the $D_{1}$-classical linear functionals, as $D_{1}%
$-semiclassical of class $s=0.$ It will be very useful in the sequel taking
into account that most of the semiclassical linear functionals of class $s=1$
are related to them. Indeed, in Section 4, the classification of such
semiclassical linear functionals is given. Some of them are not known in the
literature, as far as we know. Finally, in Section 5, limit relations in terms
of their parameters for semi-classical orthogonal polynomials are studied.

\section{Preliminaries and basic background}

\begin{definition}
Let $\left\{  \mu_{n}\right\}  _{n\geq0} $ be a sequence of complex numbers
and let $\mathcal{L}$ be a linear complex valued function defined on the
linear space $\mathbb{P}$ of polynomials with complex coefficients by
\[
\left\langle \mathcal{L},x^{n}\right\rangle =\mu_{n}.
\]
Then, $\mathcal{L}$ is called the moment functional determined by the moment
sequence $\left\{  \mu_{n}\right\}  _{n\geq0} $ and $\mu_{n}$ is called the
moment of order $n$.
\end{definition}

Given a moment functional $\mathcal{L},$ the formal Stieltjes function of
$\mathcal{L}$ is defined by%
\[
S_{\mathcal{L}}(z)=-\sum\limits_{n=0}^{\infty}\frac{\mu_{n}}{z^{n+1}}.
\]
For any moment functional $\mathcal{L}$ and any polynomial $q(x),$ we define
the moment functional $q\mathcal{L}$ by%
\[
\left\langle q\mathcal{L},P\right\rangle =\left\langle \mathcal{L}%
,qP\right\rangle ,\quad P\in\mathbb{P}.
\]

In the sequel, we will denote the set of nonnegative integers by
$\mathbb{N}_{0}.$

\begin{definition}
Let $\mathcal{L}$ be the linear functional associated with the moment sequence
$\left\{  \mu_{n}\right\}  _{n\geq0} $ and
\[
\Delta_{n}=\det%
\begin{bmatrix}
\mu_{0} & \mu_{1} & \cdots & \mu_{n}\\
\mu_{1} & \mu_{2} & \cdots & \mu_{n+1}\\
\vdots & \vdots & \ddots & \vdots\\
\mu_{n} & \mu_{n+1} & \cdots & \mu_{2n}%
\end{bmatrix}
.
\]
Then, $\mathcal{L}$ is said to be regular/quasidefinite (resp. positive
definite) if $\Delta_{n}\neq0$ (resp. $>0$) for all $n\in\mathbb{N}_{0}.$
\end{definition}

\begin{definition}
A sequence of polynomials $\left\{  P_{n}\left(  x\right)  \right\}  _{n\geq
0},$ $\deg\left(  P_{n}\right)  =n,$ is said to be an orthogonal polynomial
sequence with respect to a regular linear functional $\mathcal{L},$ if there
exists a sequence of nonzero real numbers $\{\zeta_{n}\}_{n\geq0}$ such that
\[
\left\langle \mathcal{L},P_{k}P_{n}\right\rangle =\zeta_{n}\delta_{k,n},\quad
k,n\in\mathbb{N}_{0}.
\]

\end{definition}

If $\zeta_{n}=1,$ then $\left\{  p_{n}\left(  x\right)  \right\}  _{n\geq0}$
is said to be an orthonormal polynomial sequence. Notice that if the linear
functional is positive definite there exists a unique sequence of orthonormal
polynomials assuming the leading coefficient is a positive real number.

If the leading coefficient of $P_{n}(x)$ is $1$ for every $n\in\mathbb{N}_{0}%
$, then the sequence is said to be a monic orthogonal sequence (MOPS, in short).

\begin{theorem}
Let $\left\{  b_{n}\right\}  _{n\geq0} $ and $\left\{  \gamma_{n}\right\}
_{n\geq0}$ with $\gamma_{n}\neq0$ for every $n\in\mathbb{N}_{0}$ be arbitrary
sequences of complex numbers and let $\left\{  P_{n}\left(  x\right)
\right\}  $ a sequence of monic polynomials defined by the three term
recurrence formula (TTRR)
\begin{equation}
P_{n+1}(x)=\left(  x-b_{n}\right)  P_{n}(x)-\gamma_{n}P_{n-1}(x),
\label{3term}%
\end{equation}
with $P_{-1}=0$ and $P_{0}=1.$ Then, there is a unique linear functional
$\mathcal{L}$ such that $\mathcal{L}\left(  1\right)  =\gamma_{0}$ and
\[
\left\langle \mathcal{L}, P_{k}\left(  x\right)  P_{n}\left(  x\right)
\right\rangle =\gamma_{0}\gamma_{1}\cdots\gamma_{n}\delta_{k,n}.
\]

\end{theorem}

\begin{proof}
See \cite[Theorem 4.4]{MR0481884}
\end{proof}

If the linear functional is positive definite and $\left\{  p_{n}\left(
x\right)  \right\}  _{n\geq0} $ is the corresponding orthonormal polynomial
sequence, then the above TTRR formula becomes
\[
a_{n+1}p_{n+1}(x)=\left(  x-b_{n}\right)  p_{n}(x)-a_{n}p_{n-1}(x),
\]
where $a_{n}$ is a real number and $a_{n}^{2}=\gamma_{n}.$

\begin{definition}
Let $\mathcal{L}$ be a linear functional and $U^{\ast}:\mathbb{P\rightarrow
P}$ be a linear operator. The linear functional $U\mathcal{L}$ is defined by%
\[
\left\langle U\mathcal{L},P\right\rangle =-\left\langle \mathcal{L},U^{\ast
}P\right\rangle ,\quad P\in\mathbb{P}.
\]

\end{definition}

\begin{example}
If $U$ is the standard derivative operator $D$, we have $U^{\ast}=U=D.$
\end{example}

\begin{definition}
A regular linear functional $\mathcal{L}$ is called $U$-semiclassical if it
satisfies the Pearson equation $U\left(  \phi\mathcal{L}\right)
+\psi\mathcal{L}=0$ or, equivalently,%

\[
\left\langle U\left(  \phi\mathcal{L}\right)  +\psi\mathcal{L},P\right\rangle
=0,\quad P\in\mathbb{P},
\]
where $\phi,\psi$ are two polynomials, and $\phi$ is monic. The corresponding
orthogonal sequence $\left\{  P_{n}\left(  x\right)  \right\}  _{n\geq0} $ is
called $U$-semiclassical.
\end{definition}

In the literature, semiclassical linear functionals with respect to different
choices of operators have been studied. In particular, if $U=D$ (the standard
derivative operator), the theory of $D$-semiclassical linear functionals has
been exhaustively studied by P. Maroni and co-workers (see \cite{MR1270222}
for an excellent survey on this topic).

If $U=D_{\omega},$ where
\[
D_{\omega}f(x)=\frac{f(x+\omega)-f(x)}{\omega},\quad\omega\neq0,
\]
a regular linear functional $\mathcal{L}$ is said to be $D_{\omega}%
$-semiclassical if there exist polynomials $\phi,\psi,$ where $\phi$ is monic
and deg$\psi\geq1$ such that $D_{\omega}\left(  \phi\mathcal{L}\right)
+\psi\mathcal{L}=0.$

Notice that
\[
D_{1}f(x)=f(x+1)-f(x)=\Delta f(x),
\]%
\[
D_{-1}f(x)=f(x)-f(x-1)=\nabla f(x),
\]
are the forward and backward difference operators, respectively, and
\[
\underset{\omega\rightarrow0}{\lim}D_{\omega}f(x)=Df(x)=f^{\prime}(x).
\]
If $U=D_{\omega},$ we define $U^{\ast}=D_{-\omega}.$ With this definition, we
have $\Delta^{\ast}=\nabla$ and when $\omega\rightarrow0$ we recover the
identity $U^{\ast}=D=U.$

The concept of class of a $D_{\omega}$-semiclassical linear functional plays a
central role in order to give a constructive theory of such linear functionals.

\begin{definition}
If $\mathcal{L}$ is a $D_{\omega}$- semiclassical linear functional, then the
class $s$ of $\mathcal{L}$ is defined by%
\[
s=\underset{\phi,\psi}{~\min}\max\left\{  \deg\phi-2,\deg\psi-1\right\}  .
\]
among all polynomials $\phi,\psi$ such that the Pearson equation holds. Notice
that the class $s$ is always nonnegative.
\end{definition}

For any complex number $c,$ we introduce the linear application $\theta
_{c}:\mathbb{P\rightarrow P}$ defined by
\[
\theta_{c}(p)(x)=\frac{p(x)-p(c)}{x-c}.
\]
We have the following result \cite{MR1270222}.

\begin{theorem}
The regular linear functional $\mathcal{L}$ satisfying the Pearson equation
\[
D_{\omega}\left(  \phi\mathcal{L}\right)  +\psi\mathcal{L}=0
\]
is of class $s$ if and only if%
\[
\prod_{c\in Z(\phi)}\left(  |\psi(c-\omega)+(\theta_{c}\phi)(c-\omega
)|+|\left\langle \mathcal{L},\theta_{c-\omega}(\psi+\theta_{c}\phi
)\right\rangle |\right)  >0,
\]
where $Z(\phi)$ denotes the set of zeros of the polynomial $\phi(x)$.

When there exists $c\in{Z(\phi)}$ such that%
\[
\psi(c-\omega)+(\theta_{c}\phi)(c-\omega)=\left\langle \mathcal{L}%
,\theta_{c-\omega}(\psi+\theta_{c}\phi)\right\rangle =0,
\]
the Pearson equation becomes
\[
D_{\omega}\left[  (\theta_{c}\phi)\mathcal{L}\right]  +\left[  \theta
_{c-\omega}(\psi+\theta_{c}\phi)\right]  \mathcal{L}=0.
\]

\end{theorem}

\begin{remark}
When $s=0$, the $D_{\omega}$-classical orthogonal polynomials appear (see
\cite{MR1464669}). When $\omega=1$ several characterizations of classical
orthogonal polynomials have been done in \cite{MR1340932}. Indeed, we explain
with more detail in the next section the main characteristics of these
polynomials and their corresponding linear functionals.

The $D_{1}$-semiclassical linear functionals have been studied by F.
Marcell\'{a}n-Salto and they are characterized following the same ideas as in
the $D$ case. Maroni-Mejri deduced the Laguerre-Freud equations for the
coefficients of the TTRR of $D_{w}$- semiclassical orthogonal polynomials of
class $s=1$. In the symmetric case, i.e., the moments of odd order vanish,
they deduce the explicit values of such coefficients and the integral
representations of the corresponding linear functionals are given.
\end{remark}

On the other hand, the Pearson equation yields a difference equation for the
moments of the linear functional and, as consequence, we get a linear
difference equation with polynomials coefficients satisfied by the Stieltjes
function associated with the linear functional. Indeed,

\begin{theorem}
If $\mathcal{L}$ is a $D_{\omega}$-semiclassical moment functional, then the
formal Stieltjes function of $\mathcal{L}$ satisfies a non homogeneous first
order linear difference equation%
\[
\phi\left(  z\right)  D_{\omega} S_{\mathcal{L}}(z)=a(z)S_{\mathcal{L}%
}(z)+b(z),
\]
where $a(z)$ and $b(z)$ are polynomials depending on $\phi$ and $\psi$, with
$\deg\left(  a\right)  \leq s+1$ and $\deg(b)\leq s.$
\end{theorem}

\section{Discrete semiclassical orthogonal polynomials}

In the sequel, we will consider linear functionals
\[
\left\langle \mathcal{L},P\right\rangle =\sum\limits_{x=0}^{\infty}%
P(x)\rho\left(  x\right)  ,
\]
for some positive weight function $\rho\left(  x\right)  $ supported on a
countable subset of the real line. With this choice, the Pearson equation
\[
\left\langle \Delta\left(  \phi\mathcal{L}\right)  +\psi\mathcal{L}%
,P\right\rangle =0,\quad P\in\mathbb{P},
\]
yields%
\begin{equation}
\Delta\left(  \phi\rho\right)  +\psi\rho=0. \label{Pearson}%
\end{equation}
We rewrite this equation as
\[
\frac{\rho\left(  x+1\right)  }{\rho\left(  x\right)  }=\frac{\phi\left(
x\right)  -\psi\left(  x\right)  }{\phi\left(  x+1\right)  }=\frac
{\lambda\left(  x\right)  }{\phi\left(  x+1\right)  },
\]
with%
\[
\phi\left(  x\right)  =x\left(  x+\beta_{1}\right)  \left(  x+\beta
_{2}\right)  \cdots\left(  x+\beta_{r}\right)  ,\quad r\leq s+2,
\]
and%
\[
\lambda\left(  x\right)  =c\left(  x+\alpha_{1}\right)  \left(  x+\alpha
_{2}\right)  \cdots\left(  x+\alpha_{l}\right)  ,\quad l\leq s+2.
\]

Since the Pochhammer symbol $\left(  \alpha\right)  _{x}$ defined by%
\[
\left(  \alpha\right)  _{x}=\frac{\Gamma\left(  \alpha+x\right)  }%
{\Gamma\left(  \alpha\right)  }, x\in\mathbb{N}_{0},
\]
satisfies the identity%
\[
\frac{\left(  \alpha\right)  _{x+1}}{\left(  \alpha\right)  _{x}}=x+\alpha,
\]
we obtain%
\[
\rho\left(  x\right)  =\frac{\left(  \alpha_{1}\right)  _{x}\cdots\left(
\alpha_{l}\right)  _{x}}{\left(  \beta_{1}+1\right)  _{x}\cdots\left(
\beta_{r}+1\right)  _{x}}\frac{c^{x}}{x!}.
\]
We will denote the orthogonal polynomials associated with $\rho\left(
x\right)  $ by%
\[
P_{n}^{(l,r)}\left(  x;\alpha_{1},\ldots,\alpha_{l},\beta_{1},\ldots,\beta
_{r};c\right)  .
\]

\subsection{\vspace{0in}Discrete classical polynomials}

When $s=0,$ we solve the Pearson equation (\ref{Pearson}) with $\deg\left(
\phi\right)  \leq2$, $\deg\left(  \psi\right)  = 1$ and three canonical cases
appear \cite{MR1149380}:

\begin{enumerate}
\item $P_{n}^{(0,0)}\left(  x;c\right)  $ Charlier polynomials

We have%
\begin{equation}
\rho\left(  x\right)  =\frac{c^{x}}{x!},\quad c>0,\quad x\in\mathbb{N}%
_{0},\label{Charlier}%
\end{equation}%
\begin{equation}
\phi\left(  x\right)  =x,\quad\psi\left(  x\right)
=x-c,\label{PearsonCharlier}%
\end{equation}
and%
\[
\lambda(x)=c.
\]

The Charlier polynomials have the hypergeometric representation
\cite{MR2656096}%
\begin{equation}
C_{n}\left(  x;c\right)  =\ _{2}F_{0}\left(  -n,-x;;-\frac{1}{c}\right)
\label{CharlierHyper}%
\end{equation}
and the monic Charlier polynomials $\widehat{C}_{n}\left(  x;c\right)  $
become%
\begin{equation}
\widehat{C}_{n}\left(  x;c\right)  =\left(  -c\right)  ^{n}C_{n}\left(
x;c\right)  . \label{CharlierMonic}%
\end{equation}

\item $P_{n}^{(1,0)}\left(  x;\alpha;c\right)  $ Meixner polynomials

We have
\begin{equation}
\rho\left(  x\right)  =\left(  \alpha\right)  _{x}\frac{c^{x}}{x!},\quad
\alpha>0,\quad0<c<1,\quad x\in\mathbb{N}_{0},\label{Meixner}%
\end{equation}%
\begin{equation}
\phi\left(  x\right)  =x,\quad\psi\left(  x\right)  =\left(  1-c\right)
x-c\alpha,\label{PearsonMeixner}%
\end{equation}
and%
\[
\lambda\left(  x\right)  =c\left(  x+\alpha\right)  .
\]

The Meixner polynomials have the hypergeometric representation
\cite{MR2656096}%
\begin{equation}
M_{n}\left(  x;\alpha,c\right)  =\ _{2}F_{1}\left(  -n,-x;\alpha;1-\frac{1}%
{c}\right)  \label{MeixnerHyper}%
\end{equation}
and the monic Meixner polynomials $\widehat{M}_{n}\left(  x;\alpha,c\right)  $
become%
\begin{equation}
\widehat{M}_{n}\left(  x;\alpha,c\right)  =\left(  \alpha\right)  _{n}\left(
\frac{c}{c-1}\right)  ^{n}M_{n}\left(  x;\alpha,c\right)  .
\label{MeixnerMonic}%
\end{equation}

If $\alpha=-N,$ we obtain the Krawtchouk polynomials $K_{n}\left(
x;N;c\right)  $ where%
\begin{equation}
\rho\left(  x\right)  =\left(  -N\right)  _{x}\frac{c^{x}}{x!},\quad c<0,\quad
N\in\mathbb{N},\quad x\in\ \left[  0,N\right]  ,\label{Krawtchouk}%
\end{equation}
and%
\begin{equation}
\phi\left(  x\right)  =x,\quad\psi\left(  x\right)  =\left(  1-c\right)
x+cN.\label{PearsonKraw}%
\end{equation}

\item $P_{n}^{(2,1)}\left(  x;\alpha, -N,\beta;1\right)  $ Hahn polynomials

We have%
\begin{equation}
\rho\left(  x\right)  =\frac{\left(  \alpha\right)  _{x}\left(  -N\right)
_{x}}{\left(  \beta+1\right)  _{x}}\frac{1}{x!},\quad x\in\ \left[
0,N\right]  ,\label{Hahn}%
\end{equation}
where $\alpha,\beta+1\notin\left[  -N,0\right]  ,$ $\alpha\left(
\beta+1\right)  <0,$ $N\in\mathbb{N},$%
\begin{equation}
\phi\left(  x\right)  =x\left(  x+\beta\right)  ,\quad\psi\left(  x\right)
=\left(  \beta-\alpha+N\right)  x+\alpha N,\label{PearsonHahn}%
\end{equation}
and%
\[
\lambda\left(  x\right)  =\left(  x+\alpha\right)  \left(  x-N\right)  .
\]
The Hahn polynomials have the hypergeometric representation \cite{MR2656096}%
\begin{equation}
Q_{n}\left(  x;\alpha,\beta,N\right)  =\ _{3}F_{2}\left(  -n,-x,n-1-\beta
+\alpha-N;\alpha,-N;1\right)  .\label{HahnHyper}%
\end{equation}

\end{enumerate}

\section{Discrete semi-classical polynomials of class 1}

When $s=1,$ we solve the Pearson equation (\ref{Pearson}) with $\deg\left(
\phi\right)  \leq3$, $\deg\left(  \psi\right)  \leq2,$ and obtain the
following canonical cases:

\begin{enumerate}
\item $P_{n}^{(0,1)}\left(  x;\beta;c\right)  $ Generalized Charlier
polynomials%
\begin{equation}
\rho\left(  x\right)  =\frac{1}{\left(  \beta+1\right)  _{x}}\frac{c^{x}}%
{x!},\quad x\in\mathbb{N}_{0},\label{GCharlier}%
\end{equation}
where $\beta>-1$ and $c>0.$

\item $P_{n}^{(1,1)}\left(  x;\alpha,\beta;c\right)  $ Generalized Meixner
polynomials
\begin{equation}
\rho\left(  x\right)  =\frac{\left(  \alpha\right)  _{x}}{\left(
\beta+1\right)  _{x}}\frac{c^{x}}{x!},\quad x\in\mathbb{N}_{0}%
,\label{GMeixner}%
\end{equation}
where $\alpha\left(  \beta+1\right)  >0$ and $c>0.$

\item $P_{n}^{(2,0)}\left(  x;\alpha,-N;1\right)  $ "Generalized Krawtchouk
polynomials"%
\[
\rho\left(  x\right)  =\left(  \alpha\right)  _{x}\left(  -N\right)  _{x}%
\frac{1}{x!},\quad x\in\ \left[  0,N\right]  ,
\]
where $\alpha<-N$ and $N\in\mathbb{N}.$

\item $P_{n}^{(2,1)}\left(  x;\alpha_{1},\alpha_{2};\beta;c\right)  $
"Generalized Hahn polynomials of type I"%
\begin{equation}
\rho\left(  x\right)  =\frac{\left(  \alpha_{1}\right)  _{x}\left(  \alpha
_{2}\right)  _{x}}{\left(  \beta+1\right)  _{x}}\frac{c^{x}}{x!},\quad
x\in\mathbb{N}_{0},\label{GHahnI}%
\end{equation}
where $\alpha_{1}\alpha_{2}\left(  \beta+1\right)  >0$ and $0<c<1.$

\item $P_{n}^{(3,2)}\left(  x;\alpha_{1},\alpha_{2},-N,\beta_{1},\beta
_{2},;1\right)  $ "Generalized Hahn polynomials of type II"%
\begin{equation}
\rho\left(  x\right)  =\frac{\left(  \alpha_{1}\right)  _{x}\left(  \alpha
_{2}\right)  _{x}\left(  -N\right)  _{x}}{\left(  \beta_{1}+1\right)
_{x}\left(  \beta_{2}+1\right)  _{x}}\frac{1}{x!},\quad x\in\ \left[
0,N\right]  \label{GHahnII}%
\end{equation}
where $\alpha_{1},\alpha_{2}\notin\left[  -N,0\right]  ,$ $\beta_{1},\beta
_{2}\notin\left[  -N-1,-1\right]  ,$ $\alpha_{1}\alpha_{2}\left(  \beta
_{1}+1\right)  \left(  \beta_{2}+1\right)  <0$ and $N\in\mathbb{N}.$
\end{enumerate}

In the next sections, we study these polynomials in detail.

\subsection{Generalized Charlier polynomials}

We have%
\[
\phi\left(  x\right)  =x\left(  x+\beta\right)  ,\quad\psi\left(  x\right)
=x^{2}+\beta x-c,
\]
and%
\[
\lambda\left(  x\right)  =c.
\]

The first moments are
\[
\mu_{0}=c^{-\frac{\beta}{2}}I_{\beta}\left(  2\sqrt{c}\right)  \Gamma\left(
\beta+1\right)  ,\quad\mu_{1}=c^{\frac{1-\beta}{2}}I_{\beta+1}\left(
2\sqrt{c}\right)  \Gamma\left(  \beta+1\right)  ,
\]
where $I_{\nu}\left(  z\right)  $ is the modified Bessel function of the first
kind \cite[10.25.2]{MR2723248}.

In \cite{MR1737084} Hounkonnou, Hounga and Ronveaux studied the semiclassical
polynomials associated with the weight function%
\begin{equation}
\rho_{r}\left(  x\right)  =\frac{c^{x}}{\left(  x!\right)  ^{r}},\quad
r=0,1,\ldots. \label{GenCharlier}%
\end{equation}
When $r=2,$ they derive the Laguerre-Freud equations for the recurrence
coefficients and a second-order difference equation.

In \cite{MR2063533} Van Assche and Foupouagnigni also consider
(\ref{GenCharlier}) with $r=2.$ They simplify the Laguerre-Freud equations
obtained in \cite{MR1737084} and get%
\begin{align*}
u_{n+1}+u_{n-1}  &  =\frac{1}{\sqrt{c}}\frac{nu_{n}}{1-u_{n}^{2}},\\
v_{n}  &  =\sqrt{c}u_{n+1}u_{n}%
\end{align*}
with $\gamma_{n}=c\left(  1-u_{n}^{2}\right)  $ and $\beta_{n}=v_{n}+n$. They
show that these equations are related to the discrete Painlev\'{e} II equation
dP$_{\mathrm{II}}$ \cite{MR2451211}
\[
x_{n+1}+x_{n-1}=\frac{\left(  an+b\right)  x_{n}+c}{1-x_{n}^{2}}.
\]
They also obtain the asymptotic behavior%
\[
\underset{n\rightarrow\infty}{\lim}\gamma_{n}=c,\quad\underset{n\rightarrow
\infty}{\lim}v_{n}=0,
\]
and conclude that the asymptotic zero distribution is given by the uniform
distribution on $[0,1]$, as is the case for the usual Charlier polynomials
\cite{MR1687531}.

In \cite{MR2926308}, Smet and Van Assche studied the orthogonal polynomials
associated with the weight function (\ref{GCharlier}). They obtained the
Laguerre-Freud equations
\begin{align}
\left(  a_{n+1}^{2}-c\right)  \left(  a_{n}^{2}-c\right)   &  =c\left(
b_{n}-n\right)  \left(  b_{n}-n+\beta\right)  ,\label{LFCharlier}\\
b_{n}+b_{n-1}  &  =n-1-\beta+\frac{cn}{a_{n}^{2}},\nonumber
\end{align}
for the orthonormal polynomials. They showed that these equations are are a
limiting case of the discrete Painlev\'{e} IV equation dP$_{\mathrm{IV}}$
\cite{MR2451211}%
\begin{align*}
x_{n+1}x_{n}  &  =\frac{\left(  y_{n}-\delta n-E\right)  ^{2}-A}{y_{n}^{2}%
-B},\\
y_{n}+y_{n-1}  &  =\frac{\delta n+E-\delta/2-C}{1+Dx_{n}}+\frac{\delta
n+E-\delta/2+C}{1+x_{n}/D}.
\end{align*}

Finally, in \cite{MR2944743} Filipuk and Van Assche related the system
\ref{LFCharlier} to the (continuous) fifth Painlev\'{e} equation
P$_{\mathrm{V}}$.

\subsection{Generalized Meixner polynomials}

We have%
\[
\phi\left(  x\right)  =x\left(  x+\beta\right)  ,\quad\psi\left(  x\right)
=x^{2}+\left(  \beta-c\right)  x-c\alpha
\]
and%
\[
\lambda\left(  x\right)  =c\left(  x+\alpha\right)  .
\]

The first moments are
\[
\mu_{0}=M\left(  \alpha,\beta+1;c\right)  ,\quad\mu_{1}=\frac{\alpha c}%
{\beta+1}M\left(  \alpha+1,\beta+2;c\right)  ,
\]
where $M\left(  a,b;z\right)  $ is the confluent hypergeometric function
\cite[13.2.2]{MR2723248}.

In \cite{MR842801} Ronveaux considered the semiclassical polynomials
associated with the weight function%
\[
\rho_{r}\left(  x\right)  =\prod\limits_{j=1}^{r}\left(  \alpha_{j}\right)
_{x}\ \frac{c^{x}}{\left(  x!\right)  ^{r}},\quad r=1,2,\ldots,
\]
and in \cite{MR1780044} he made some conjectures on the asymptotic behavior of
the recurrence coefficients.

In \cite{MR2926308}, Smet and Van Assche studied the orthogonal polynomials
associated with the weight function (\ref{GMeixner}). They obtained the
Laguerre-Freud equations
\begin{align}
\left(  u_{n}+v_{n}\right)  \left(  u_{n+1}+v_{n}\right)   &  =\frac{\alpha
-1}{c^{2}}v_{n}\left(  v_{n}-c\right)  \left(  v_{n}-c\frac{\alpha-1-\beta
}{\alpha-1}\right)  ,\label{LFMeixner}\\
\left(  u_{n}+v_{n}\right)  \left(  u_{n+1}+v_{n-1}\right)   &  =\frac{u_{n}%
}{u_{n}-\frac{cn}{\alpha-1}}\left(  u_{n}+c\right)  \left(  u_{n}%
+c\frac{\alpha-1-\beta}{\alpha-1}\right)  ,\nonumber
\end{align}
for the orthonormal polynomials, with
\begin{align*}
a_{n}^{2}  &  =cn-\left(  \alpha-1\right)  u_{n},\\
b_{n}  &  =n+\alpha+c-\beta-1-\frac{\alpha-1}{c}v_{n}.
\end{align*}
They also proved that the system (\ref{LFMeixner}) is a limiting case of the
asymmetric discrete Painlev\'{e} IV equation $\alpha-$dP$_{\mathrm{IV}}$
\cite{MR2451211}.

In \cite{MR2861208} Filipuk and Van Assche showed that the system
(\ref{LFMeixner}) can be obtained from the B\"{a}cklund transformation of the
fifth Painlev\'{e} equation P$_{\mathrm{V}}$. The particular case of
(\ref{GMeixner}) when $\beta=0$ was considered by Boelen, Filipuk, and Van
Assche in \cite{MR2749070}.

If we set $\alpha=-N,$ $N\in\mathbb{N}$ in (\ref{GMeixner}), we obtain%
\[
\rho\left(  x\right)  =\frac{\left(  -N\right)  _{x}}{\left(  \beta+1\right)
_{x}}\frac{c^{x}}{x!},
\]
where we now have $\beta>-1$ and $c<0.$ This case was analyzed by Boelen,
Filipuk, Smet, Van Assche, and Zhang in \cite{Krawtchouk}.

\subsubsection{Singular limits}

If we let $\alpha\rightarrow0$ and $\beta\rightarrow-1$ in (\ref{GMeixner}),
we have $\rho\left(  x\right)  \rightarrow\widetilde{\rho}\left(  x\right)  $
where $\widetilde{\rho}\left(  x\right)  $ is a new weight function satisfying
the Pearson equation
\begin{equation}
\Delta\left[  \left(  x-1\right)  x\widetilde{\rho}\right]  +\left[  x-\left(
c+1\right)  \right]  x\widetilde{\rho}=0. \label{LimitMeixner1}%
\end{equation}
Assuming that $\widetilde{\rho}\left(  x\right)  $ satisfies $x\widetilde{\rho
}(x)=xu(x),$ for some weight function $u(x),$ we get%
\begin{equation}
\Delta\left[  \left(  x-1\right)  xu\right]  +\left[  x-\left(  c+1\right)
\right]  xu=0 . \label{LimitMeixner2}%
\end{equation}

Using the product rule
\begin{equation}
\Delta\left(  fg\right)  =f\Delta g+g\Delta f+\Delta f\Delta g
\label{ProductRule}%
\end{equation}
in (\ref{LimitMeixner2}), we have%
\[
xu+\left(  x-1\right)  \Delta\left(  xu\right)  +\Delta\left(  xu\right)
+\left[  x-\left(  c+1\right)  \right]  xu=0,
\]
or%
\[
x\Delta\left(  xu\right)  +\left[  x-\left(  c+1\right)  +1\right]  xu=0.
\]
Dividing by $x,$ we obtain%
\[
\Delta\left(  xu\right)  +\left(  x-c\right)  u=0.
\]
Comparing with (\ref{PearsonCharlier}), we see that $u(x)$ is the weight
function corresponding to the Charlier polynomials (\ref{Charlier}), and
therefore (\ref{LimitMeixner1}) implies that%
\begin{equation}
\widetilde{\rho}\left(  x\right)  =\frac{c^{x}}{x!}+M\delta\left(  x\right)  ,
\label{DiracCharlier1}%
\end{equation}
where $\delta\left(  x\right)  $ is the Dirac delta function.

The orthogonal polynomials $P_{n}^{(1,1)}\left(  x;0,-1;c\right)  $ associated
with the weight function (\ref{DiracCharlier1}) were first studied by Chihara
in \cite{MR813269}. He showed that they satisfy a 3-term recurrence relation
(\ref{3term}) with
\[
b_{n}=c\frac{n}{n+1}\frac{D_{n}}{D_{n+1}}+\left(  n+1\right)  \frac{D_{n+1}%
}{D_{n}},
\]
and%
\[
\gamma_{n}=c\frac{n^{2}}{n+1}\frac{D_{n}^{2}}{D_{n-1}D_{n+1}},
\]
where%
\[
D_{n}=\frac{c^{n}}{n!}\frac{M}{e^{c}+MK_{n-1}},
\]
and%
\[
K_{n}=\sum\limits_{j=0}^{n}\frac{c^{j}}{j!},\quad K_{-1}=0.
\]
Note that in order that $D_{n}$ is well defined for all $n,$ we need $M>-1,$
since $K_{n}\nearrow e^{c}.$

In \cite{MR1327166}, Bavinck and Koekoek obtained a difference equation
satisfied by these polynomials and in \cite{MR1379116} \'{A}lvarez-Nodarse,
Garc\'{\i}a, and Marcell\'{a}n found the hypergeometric representation%
\[
P_{n}^{(1,1)}\left(  x;0;-1;c\right)  =\left(  -c\right)  ^{n}\ _{3}%
F_{1}\left(  -n,-x,1+\frac{x}{D_{n}};\frac{x}{D_{n}};-\frac{1}{c}\right)  .
\]
Since%
\[
\underset{z\rightarrow\infty}{\lim}\frac{\left(  1+z\right)  _{x}}{\left(
z\right)  _{x}}=1,
\]
we see that%
\[
\underset{M\rightarrow0}{\lim}P_{n}^{(1,1)}\left(  x;0,-1;c\right)
=\widehat{C}_{n}\left(  x;c\right)  ,
\]
where $\widehat{C}_{n}\left(  x;c\right)  $ is the monic Charlier polynomial
(\ref{CharlierMonic}).

\subsection{Generalized Krawtchouk polynomials}

We have%
\[
\phi\left(  x\right)  =x,\quad\psi\left(  x\right)  =-c\left[  x^{2}+\left(
\alpha-N-\frac{1}{c}\right)  x-\alpha N\right]
\]
and%
\[
\lambda\left(  x\right)  =c\left(  x+\alpha\right)  \left(  x-N\right)  .
\]

The first moments are%
\[
\mu_{0}=C_{N}\left(  -\alpha;-\frac{1}{c}\right)  ,\quad\mu_{1}=-c\alpha
NC_{N-1}\left(  -\alpha-1;-\frac{1}{c}\right)  ,
\]
where $C_{n}\left(  x;a\right)  $ is the Charlier polynomial
(\ref{CharlierHyper}).

To our knowledge, these polynomials have not appeared before in the literature.

\subsection{Generalized Hahn polynomials of type I}

We have%
\[
\phi\left(  x\right)  =x\left(  x+\beta\right)  ,
\]%
\[
\psi\left(  x\right)  =\left(  1-c\right)  x^{2}+\left(  \beta-c\alpha
_{1}-c\alpha_{2}\right)  x-c\alpha_{1}\alpha_{2},
\]
and%
\[
\lambda\left(  x\right)  =c\left(  x+\alpha_{1}\right)  \left(  x+\alpha
_{2}\right)  .
\]

The first moments are%
\begin{align*}
\mu_{0}  &  =\ _{2}F_{1}\left(  \alpha_{1},\alpha_{2};\beta+1,c\right)  ,\\
\mu_{1}  &  =c\frac{\alpha_{1}\alpha_{2}}{\beta+1}\ _{2}F_{1}\left(
\alpha_{1}+1,\alpha_{2}+1;\beta+2,c\right)  ,
\end{align*}
$\ $where $_{2}F_{1}\left(  a,b;c;z\right)  $ is the hypergeometric function.

\subsubsection{Singular limits}

If we let $\alpha_{2}\rightarrow0,\beta\rightarrow-1$ and $\alpha_{1}=\alpha$
in (\ref{GHahnI}), we have $\rho\left(  x\right)  \rightarrow\widetilde{\rho
}\left(  x\right)  $ where $\widetilde{\rho}\left(  x\right)  $ is a new
weight function satisfying the Pearson equation
\begin{equation}
\Delta\left[  \left(  x-1\right)  x\widetilde{\rho}\right]  +\left[  \left(
1-c\right)  x-\left(  1+c\alpha\right)  \right]  x\widetilde{\rho}=0.
\label{LimitHahnI1}%
\end{equation}
Assuming that $\widetilde{\rho}\left(  x\right)  $ satisfies $x\widetilde{\rho
}(x)=xu(x),$ for some weight function $u(x),$ we get%
\begin{equation}
\Delta\left[  \left(  x-1\right)  xu\right]  +\left[  \left(  1-c\right)
x-\left(  1+c\alpha\right)  \right]  xu=0. \label{LimitHahnI2}%
\end{equation}

Using the product rule (\ref{ProductRule}) in (\ref{LimitHahnI2}), we have%
\[
xu+\left(  x-1\right)  \Delta\left(  xu\right)  +\Delta\left(  xu\right)
+\left[  \left(  1-c\right)  x-\left(  1+c\alpha\right)  \right]  xu=0,
\]
or%
\[
x\Delta\left(  xu\right)  +\left[  \left(  1-c\right)  x-\left(
1+c\alpha\right)  +1\right]  xu=0.
\]
Dividing by $x,$ we obtain%
\[
\Delta\left(  xu\right)  +\left[  \left(  1-c\right)  x-c\alpha\right]  u=0.
\]
Comparing with (\ref{PearsonMeixner}), we see that $u(x)$ is the weight
function corresponding to the Meixner polynomials (\ref{Meixner}), and
therefore (\ref{LimitMeixner1}) implies that%
\begin{equation}
\widetilde{\rho}\left(  x\right)  =\left(  \alpha\right)  _{x}\frac{c^{x}}%
{x!}+M\delta\left(  x\right)  . \label{DiracMeixner}%
\end{equation}

The orthogonal polynomials associated with the weight function
(\ref{DiracMeixner}) were first studied by Chihara in \cite{MR813269}. He
showed that they satisfy a 3-term recurrence relation (\ref{3term}) with
\[
b_{n}=\frac{c\left(  \alpha+n\right)  }{c-1}\frac{n}{n+1}\frac{B_{n}}{B_{n+1}%
}+\frac{n+1}{c-1}\frac{B_{n+1}}{B_{n}},
\]
and%
\[
\gamma_{n}=\frac{c}{\left(  c-1\right)  ^{2}}\frac{n^{2}\left(  \alpha
+n\right)  }{n+1}\frac{B_{n}^{2}}{B_{n-1}B_{n+1}},
\]
where%
\[
B_{n}=\frac{c^{n}\ \left(  \alpha\right)  _{n}}{\left(  1-c\right)  n!}%
\frac{M}{\left(  1-c\right)  ^{-\alpha}+MK_{n-1}},
\]
and%
\[
K_{n}=\sum\limits_{j=0}^{n}\left(  \alpha\right)  _{j}\ \frac{c^{j}}{j!},\quad
K_{-1}=0.
\]
Note that in order that $B_{n}$ is well defined for all $n,$ we need $M>-1,$
since $K_{n}\nearrow\left(  1-c\right)  ^{-\alpha}.$

In \cite{MR1270212}, Richard Askey proposed the problem of finding a
second-order difference equation satisfied by these polynomials. The problem
was solved by Bavinck and van Haeringen in \cite{MR1281521}, and in
\cite{MR1379116} \'{A}lvarez-Nodarse, Garc\'{\i}a and Marcell\'{a}n found the
hypergeometric representation%
\[
P_{n}^{(2,1)}\left(  x;\alpha,0,-1;c\right)  =\left(  \alpha\right)
_{n}\left(  \frac{c}{c-1}\right)  ^{n}\ _{3}F_{2}\left(  -n,-x,1+\frac
{x}{B_{n}};\alpha,\frac{x}{B_{n}};1-\frac{1}{c}\right)  .
\]
In this case,
\[
\underset{M\rightarrow0}{\lim}P_{n}^{(2,1)}\left(  x;\alpha,0,-1;c\right)
=\widehat{M}_{n}\left(  x;\alpha,c\right)  ,
\]
where $\widehat{M}_{n}\left(  x;\alpha,c\right)  $ is the monic Meixner
polynomial (\ref{MeixnerMonic}).

If $\alpha_{1}=-N,$ $N\in\mathbb{N},$ we can remove the restriction that
$0<c<1$ and take any $c<0,$ with $\alpha_{2}\notin\left[  -N,0\right]  ,$
$\beta\notin\left[  -N-1,-1\right]  ,$ and $\alpha_{2}\left(  \beta+1\right)
>0$. If we let $\alpha_{2}\rightarrow-\left(  N-1\right)  $ and $\beta
\rightarrow-N$, we have $\rho\left(  x\right)  \rightarrow\widetilde{\rho
}\left(  x\right)  $ where $\widetilde{\rho}\left(  x\right)  $ is a new
weight function satisfying the Pearson equation%
\[
\psi\left(  x\right)  =\left(  1-c\right)  x^{2}+\left(  \beta-c\alpha
_{1}-c\alpha_{2}\right)  x-c\alpha_{1}\alpha_{2},
\]%
\begin{equation}
\Delta\left[  x\left(  x-N\right)  \widetilde{\rho}\right]  +\left[  \left(
1-c\right)  x+c\left(  N-1\right)  \right]  \left(  x-N\right)
\widetilde{\rho}=0.\label{LimitHahnI3}%
\end{equation}
Assuming that $\widetilde{\rho}\left(  x\right)  $ satisfies $\left(
x-N\right)  \widetilde{\rho}(x)=\left(  x-N\right)  u(x),$ for some weight
function $u(x),$ we get%
\begin{equation}
\Delta\left[  x\left(  x-N\right)  u\right]  +\left[  \left(  1-c\right)
x+c\left(  N-1\right)  \right]  \left(  x-N\right)  u=0.\label{LimitHahnI4}%
\end{equation}

Using the product rule (\ref{ProductRule}) in (\ref{LimitHahnI4}), we have%
\[
xu+\left(  x-N+1\right)  \Delta\left(  xu\right)  +\left[  \left(  1-c\right)
x+c\left(  N-1\right)  \right]  \left(  x-N\right)  u=0,
\]
or%
\[
\left(  x-N+1\right)  \Delta\left(  xu\right)  +\left(  x-N+1\right)  \left(
x+Nc-cx\right)  u=0.
\]
Dividing by $x-N+1,$ we obtain%
\[
\Delta\left(  xu\right)  +\left[  \left(  1-c\right)  x+cN\right]  u=0.
\]
Comparing with (\ref{PearsonKraw}), we see that $u(x)$ is the weight function
corresponding to the Krawtchouk polynomials (\ref{Krawtchouk}), and therefore
(\ref{LimitHahnI3}) implies that%
\[
\widetilde{\rho}\left(  x\right)  =\left(  -N\right)  _{x}\frac{c^{x}}%
{x!}+M\delta\left(  x-N\right)  .
\]

\subsection{Generalized Hahn polynomials of type II}

We have%
\[
\phi\left(  x\right)  =x\left(  x+\beta_{1}\right)  \left(  x+\beta
_{2}\right)  ,
\]%
\begin{align*}
\psi\left(  x\right)   &  =\left(  \beta_{1}+\beta_{2}-\alpha_{1}-\alpha
_{2}+N\right)  x^{2}\\
&  +\left(  \beta_{1}\beta_{2}+\alpha_{1}N-\alpha_{1}\alpha_{2}+\alpha
_{2}N\right)  x+\alpha_{1}\alpha_{2}N,
\end{align*}
and%
\[
\lambda\left(  x\right)  =\left(  x+\alpha_{1}\right)  \left(  x+\alpha
_{2}\right)  .
\]

The first moments are%
\begin{align*}
\mu_{0}  &  =\ _{3}F_{2}\left(  \alpha_{1},\alpha_{2},-N;\beta_{1}+1,\beta
_{2}+1,1\right)  ,\\
\mu_{1}  &  =-\frac{\alpha_{1}\alpha_{2}N}{\left(  \beta_{1}+1\right)  \left(
\beta_{2}+1\right)  }\ _{3}F_{2}\left(  \alpha_{1}+1,\alpha_{2}+1,-N+1;\beta
_{1}+2,\beta_{2}+2,1\right)  ,
\end{align*}
where $_{3}F_{2}\left(  a_{1},a_{2},a_{3};b_{1},b_{2};z\right)  $ is the
hypergeometric function.

To our knowledge, these polynomials have not appeared before in the literature.

\subsubsection{Singular limits}

If we let $\alpha_{2}\rightarrow0,\beta_{2}\rightarrow-1,$ $\alpha_{1}%
=\alpha,$ $\beta_{1}=\beta,$ $\alpha\left(  \beta+1\right)  <0$ in
(\ref{GHahnII}), we have $\rho\left(  x\right)  \rightarrow\widetilde{\rho
}\left(  x\right)  $ where $\widetilde{\rho}\left(  x\right)  $ is a new
weight function satisfying the Pearson equation
\begin{equation}
\Delta\left[  \left(  x-1\right)  \left(  x+\beta\right)  x\widetilde{\rho
}\right]  +\left[  \left(  \beta-1-\alpha+N\right)  x+\alpha N-\beta\right]
x\widetilde{\rho}=0.\label{LimitHahnII}%
\end{equation}
Assuming that $\widetilde{\rho}\left(  x\right)  $ satisfies $x\widetilde{\rho
}(x)=xu(x),$ for some weight function $u(x),$ we get%
\begin{equation}
\Delta\left[  \left(  x-1\right)  \left(  x+\beta\right)  xu\right]  +\left[
\left(  \beta-1-\alpha+N\right)  x+\alpha N-\beta\right]
xu=0.\label{LimitHahnIII1}%
\end{equation}

Using the product rule (\ref{ProductRule}) in (\ref{LimitHahnIII1}), we have%
\[
\left(  x+\beta\right)  xu+x\Delta\left[  \left(  x+\beta\right)  xu\right]
+\left[  \left(  \beta-1-\alpha+N\right)  x+\alpha N-\beta\right]  xu=0,
\]
or%
\[
x\Delta\left[  \left(  x+\beta\right)  xu\right]  +\left[  \left(
\beta-\alpha+N\right)  x+\alpha N\right]  xu=0.
\]
Dividing by $x,$ we obtain%
\[
\Delta\left[  \left(  x+\beta\right)  xu\right]  +\left[  \left(  \beta
-\alpha+N\right)  x+\alpha N\right]  u=0.
\]
Comparing with (\ref{PearsonHahn}), we see that $u(x)$ is the weight function
corresponding to the Hahn polynomials (\ref{Hahn}), and therefore
(\ref{LimitHahnII}) implies that%
\begin{equation}
\widetilde{\rho}\left(  x\right)  =\frac{\left(  \alpha\right)  _{x}\left(
-N\right)  _{x}}{\left(  \beta+1\right)  _{x}}\frac{1}{x!}+M\delta\left(
x\right)  . \label{DiracHahn}%
\end{equation}

Similarly, if we let $\alpha_{2}\rightarrow-\left(  N-1\right)  ,\beta
_{2}\rightarrow-N,$ $\alpha_{1}=\alpha,$ $\beta_{1}=\beta,$ $\alpha\left(
\beta+1\right)  <0$ in (\ref{GHahnII}), we have $\rho\left(  x\right)
\rightarrow\widetilde{\rho}\left(  x\right)  $ where $\widetilde{\rho}\left(
x\right)  $ is a new weight function satisfying the Pearson equation
\begin{equation}
\Delta\left[  x\left(  x+\beta\right)  \left(  x-N\right)  \widetilde{\rho
}\right]  +\left[  \left(  \beta-\alpha+N-1\right)  x+\alpha\left(
N-1\right)  \right]  \left(  x-N\right)  \widetilde{\rho}%
=0.\label{LimitHahnIII}%
\end{equation}
Assuming that $\widetilde{\rho}\left(  x\right)  $ satisfies $\left(
x-N\right)  \widetilde{\rho}(x)=\left(  x-N\right)  u(x),$ for some weight
function $u(x),$ we get
\begin{equation}
\Delta\left[  x\left(  x+\beta\right)  \left(  x-N\right)  u\right]  +\left[
\left(  \beta-\alpha+N-1\right)  x+\alpha\left(  N-1\right)  \right]  \left(
x-N\right)  u=0.\label{LimitHahnIII2}%
\end{equation}

Using the product rule (\ref{ProductRule}) in (\ref{LimitHahnIII2}), we have%
\begin{gather*}
\left(  x+\beta\right)  xu+\left(  x-N+1\right)  \Delta\left[  \left(
x+\beta\right)  xu\right] \\
+\left[  \left(  \beta-\alpha+N-1\right)  x+\alpha\left(  N-1\right)  \right]
\left(  x-N\right)  u=0,
\end{gather*}
or%
\begin{gather*}
\left(  x-N+1\right)  \Delta\left[  \left(  x+\beta\right)  xu\right] \\
+\left(  x-N+1\right)  \left[  \left(  \beta-\alpha+N\right)  x+\alpha
N\right]  u=0.
\end{gather*}
Dividing by $x-N+1,$ we obtain%
\[
\Delta\left[  \left(  x+\beta\right)  xu\right]  +\left[  \left(  \beta
-\alpha+N\right)  x+\alpha N\right]  u=0.
\]
\strut

Comparing with (\ref{PearsonHahn}), we see that $u(x)$ is the weight function
corresponding to the Hahn polynomials (\ref{Hahn}), and therefore
(\ref{LimitHahnIII}) implies that%
\begin{equation}
\widetilde{\rho}\left(  x\right)  =\frac{\left(  \alpha\right)  _{x}\left(
-N\right)  _{x}}{\left(  \beta+1\right)  _{x}}\frac{1}{x!}+M\delta\left(
x-N\right)  . \label{DiracHahn1}%
\end{equation}
The orthogonal polynomials associated with the weight functions
(\ref{DiracHahn}) and (\ref{DiracHahn1}) were first studied by
\'{A}lvarez-Nodarse and Marcell\'{a}n in \cite{MR1410602}.

\section{Limit relations between polynomials}

From the two identities \cite{MR2656096}%
\[
\underset{\lambda\rightarrow\infty}{\lim}\left(  \lambda\alpha\right)
_{x}\left(  \frac{z}{\lambda}\right)  ^{x}=\left(  \alpha z\right)  ^{x}%
\]
and
\[
\underset{\lambda\rightarrow\infty}{\lim}\frac{1}{\left(  \lambda
\alpha\right)  _{x}}\left(  \lambda z\right)  ^{x}=\left(  \frac{z}{\alpha
}\right)  ^{x},
\]
the following limit relations follow:

\begin{enumerate}
\item Generalized Hahn polynomials of type II to generalized Hahn polynomials
of type I%
\[
\underset{N\rightarrow\infty}{\lim}P_{n}^{(3,2)}\left(  x;\alpha_{1}%
,\alpha_{2},\beta,-\frac{N}{c},N;1\right)  =P_{n}^{(2,1)}\left(  x;\alpha
_{1},\alpha_{2},\beta;c\right)  .
\]

\item Generalized Hahn polynomials of type I to generalized Krawtchouk
polynomials%
\[
\underset{\beta\rightarrow\infty}{\lim}P_{n}^{(2,1)}\left(  x;\alpha
,-N,\beta;c\beta\right)  =P_{n}^{(2,0)}\left(  x;\alpha,-N;c\right)  .
\]

\item Generalized Hahn polynomials of type I to generalized Meixner
polynomials%
\[
\underset{\alpha_{2}\rightarrow\infty}{\lim}P_{n}^{(2,1)}\left(
x;\alpha,\alpha_{2},\beta;\frac{c}{\alpha_{2}}\right)  =P_{n}^{(1,1)}\left(
x;\alpha,\beta;c\right)  .
\]

\item Generalized Meixner polynomials to generalized Charlier polynomials%
\[
\underset{\alpha\rightarrow\infty}{\lim}P_{n}^{(1,1)}\left(  x;\alpha
,\beta;\frac{c}{\alpha}\right)  =P_{n}^{(0,1)}\left(  x;\beta;c\right)  .
\]

\item Generalized Meixner polynomials to Meixner polynomials%
\[
\underset{\beta\rightarrow\infty}{\lim}P_{n}^{(1,1)}\left(  x;\alpha
,\beta;c\beta\right)  =M_{n}\left(  x;\alpha;c\right)  .
\]

\item Generalized Charlier polynomials to Charlier polynomials
\[
\underset{\beta\rightarrow\infty}{\lim}P_{n}^{(0,1)}\left(  x;\beta
;c\beta\right)  =C_{n}\left(  x;c\right)  .
\]

\end{enumerate}

We also have the singular limits:

\begin{enumerate}
\item Generalized Meixner polynomials to Charlier-Dirac polynomials%
\[
\underset{%
\begin{array}
[c]{c}%
\alpha\rightarrow0\\
\beta\rightarrow-1
\end{array}
}{\lim}P_{n}^{(1,1)}\left(  x;\alpha,\beta;c\right)  =C_{n}\left(  x;c\right)
\oplus\delta{\small (x),}%
\]

\item Generalized Hahn polynomials of type I to Hahn polynomials%
\[
\underset{%
\begin{array}
[c]{c}%
\alpha_{2}\rightarrow-N\\
c\rightarrow1
\end{array}
}{\lim}P_{n}^{(2,1)}\left(  x;\alpha,\alpha_{2},\beta;c\right)  =Q_{n}\left(
x;\alpha,\beta,N\right)  ,
\]

\item Generalized Hahn polynomials of type I to Meixner-Dirac polynomials%
\[
\underset{%
\begin{array}
[c]{c}%
\alpha_{2}\rightarrow0\\
\beta\rightarrow-1
\end{array}
}{\lim}P_{n}^{(2,1)}\left(  x;\alpha,\alpha_{2},\beta;c\right)  =M_{n}\left(
x;\alpha;c\right)  \oplus\delta{\small (x),}%
\]

\item Generalized Hahn polynomials of type I to Krawtchouk-Dirac polynomials%
\[
\underset{%
\begin{array}
[c]{c}%
\alpha_{2}\rightarrow-N+1\\
\beta\rightarrow-N
\end{array}
}{\lim}P_{n}^{(2,1)}\left(  x;-N,\alpha_{2},\beta;c\right)  =K_{n}\left(
x;-N;c\right)  \oplus\delta{\small (x-N),}%
\]

\item Generalized Hahn polynomials of type II to Hahn-Dirac polynomials%
\[
\underset{%
\begin{array}
[c]{c}%
\alpha_{2}\rightarrow0\\
\beta_{2}\rightarrow-1
\end{array}
}{\lim}P_{n}^{(3,2)}\left(  x;\alpha,\alpha_{2}, -N, \beta,\beta_{2};1\right)
=Q_{n}\left(  x;\alpha,\beta,N\right)  \oplus\delta{\small (x),}%
\]

\item Generalized Hahn polynomials of type II to Hahn-Dirac polynomials%
\[
\underset{%
\begin{array}
[c]{c}%
\alpha_{2}\rightarrow-N+1\\
\beta_{2}\rightarrow-N
\end{array}
}{\lim}P_{n}^{(3,2)}\left(  x;\alpha,\alpha_{2},-N,\beta,\beta_{2};1\right)
=Q_{n}\left(  x;\alpha,\beta,N\right)  \oplus\delta{\small (x-N),}%
\]

\end{enumerate}

where we use the notation "$\oplus\ \delta\left(  x-x_{0}\right)  "$ to denote
the addition of a delta function to the measure of orthogonality at the point
$x_{0}$.

We can summarize these results in the following scheme:

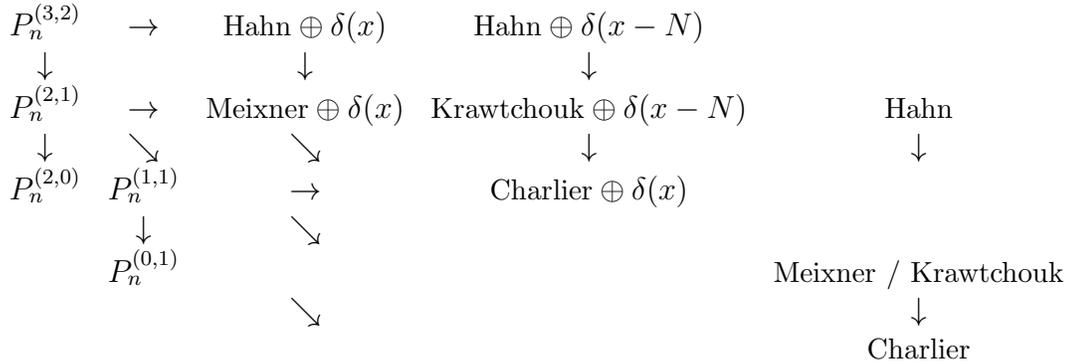
\begin{center}
$%
\begin{array}
[c]{ccccc}%
P_{n}^{(3,2)} & \rightarrow & \text{{\small Hahn}}\oplus\delta{\small (x)} &
\text{{\small Hahn}}\oplus\delta{\small (x-N)} & \\
\downarrow &  & \downarrow & \downarrow & \\
P_{n}^{(2,1)} & \rightarrow & \text{{\small Meixner}}\oplus\delta{\small (x)}
& \text{{\small Krawtchouk}}\oplus\delta{\small (x-N)} & \text{{\small Hahn}%
}\\
\downarrow & \searrow & \searrow & \downarrow & \downarrow\\
P_{n}^{(2,0)} & P_{n}^{(1,1)} & \rightarrow & \text{{\small Charlier}}%
\oplus\delta{\small (x)} & \\
& \downarrow & \searrow &  & \\
& P_{n}^{(0,1)} &  &  & \text{{\small Meixner / Krawtchouk}}\\
&  & \searrow &  & \downarrow\\
&  &  &  & \text{{\small Charlier}}%
\end{array}
$
\end{center}

\section{Concluding remarks}

We have described the discrete semiclassical orthogonal polynomials of class
$s=1$ using the different choices for the polynomials in the canonical Pearson
equation that the corresponding linear functional satisfies. We have centered
our attention when the linear functional has a representation in terms of a
discrete positive measure supported on a countable subset of the real line.
Some new families of orthogonal polynomials appear as well as some families of
orthogonal polynomials (generalized Charlier, generalized Krawtchouk, and
generalized Meixner) which have attracted the interest of researchers in the
last years taking into account the connection of the coefficients of the three
term recurrence relations with discrete and continuous Painlev\'{e} equations.
We have also studied limit relations between such families of orthogonal
polynomials having in mind an analogue of the Askey tableau for classical
orthogonal polynomials. It would be very interesting to find the equations
satisfied by the coefficients of the three term recurrence relations for the
above new sequences of semiclassical orthogonal polynomials. Furthermore, an
analysis of the class $s=2$ will also be welcome in order to get a complete
classification of such a class as well as to check if new families of
orthogonal polynomials appear as in the case of the $D$-semiclassical
orthogonal polynomials pointed out in \cite{MR2914891}.\newline

\section{Acknowledgements}

The work of the second author (FM) has been supported by Direcci\'on General
de Investigaci\'on Cient\'ifica y T\'ecnica, Ministerio de Econom\'ia y
Competitividad of Spain, grant MTM2012-36732-C03-01.\newline


\begin{thebibliography}{99}                                                                                               %


\bibitem {MR1464669}F.~Abdelkarim and P.~Maroni. \newblock The {$D_{\omega}$%
}-classical orthogonal polynomials. \newblock {\em Results Math.},
32(1-2):1--28, 1997.

\bibitem {MR1379116}R.~{\'A}lvarez-Nodarse, A.~G. Garc{\'{\i}}a, and
F.~Marcell{\'a}n. \newblock On the properties for modifications of classical
orthogonal polynomials of discrete variables. \newblock In \emph{Proceedings
of the {I}nternational {C}onference on {O}rthogonality, {M}oment {P}roblems
and {C}ontinued {F}ractions ({D}elft, 1994)}, J. Comput. Appl. Math. (65),
pages 3--18, 1995.

\bibitem {MR1353079}R.~{\'A}lvarez-Nodarse and F.~Marcell{\'a}n.
\newblock Difference equation for modifications of {M}eixner polynomials.
\newblock {\em J. Math. Anal. Appl.}, 194(1):250--258, 1995.

\bibitem {MR1410602}R.~{\'A}lvarez-Nodarse and F.~Marcell{\'a}n. \newblock The
modification of classical {H}ahn polynomials of a discrete variable.
\newblock {\em Integral Transforms Spec. Funct.}, 3(4):243--262, 1995.

\bibitem {MR1327166}H.~Bavinck and R.~Koekoek. \newblock On a difference
equation for generalizations of {C}harlier polynomials.
\newblock {\em J. Approx. Theory}, 81(2):195--206, 1995.

\bibitem {MR1281521}H.~Bavinck and H.~van Haeringen. \newblock Difference
equations for generalized {M}eixner polynomials.
\newblock {\em J. Math. Anal. Appl.}, 184(3):453--463, 1994.

\bibitem {MR1186737}S.~Belmehdi. \newblock On semi-classical linear
functionals of class {$s=1$}. {C}lassification and integral representations.
\newblock {\em Indag. Math. (N.S.)}, 3(3):253--275, 1992.

\bibitem {Krawtchouk}L.~Boelen, G.~Filipuk, C.~Smet, W.~Van~Assche, and
L.~Zhang. \newblock The generalized {K}rawtchouk polynomials and the fifth
{P}ainlev\'e equation. \newblock arXiv:1204.5070v1.

\bibitem {MR2749070}L.~Boelen, G.~Filipuk, and W.~Van~Assche.
\newblock Recurrence coefficients of generalized {M}eixner polynomials and
{P}ainlev\'e equations. \newblock {\em J. Phys. A}, 44(3):035202, 19, 2011.

\bibitem {MR1270212}C.~Brezinski, L.~Gori, and A.~Ronveaux, editors.
\newblock {\em Orthogonal polynomials and their applications}, volume~9 of
\emph{ IMACS Annals on Computing and Applied Mathematics}, Basel, 1991. J. C.
Baltzer A.G.

\bibitem {MR0481884}T.~S. Chihara.
\newblock {\em An introduction to orthogonal polynomials}. \newblock Gordon
and Breach Science Publishers, New York, 1978. \newblock Mathematics and its
Applications, Vol. 13.

\bibitem {MR813269}T.~S. Chihara. \newblock Orthogonal polynomials and
measures with end point masses. \newblock {\em Rocky Mountain J. Math.},
15(3):705--719, 1985.

\bibitem {MR2861208}G.~Filipuk and W.~Van~Assche. \newblock Recurrence
coefficients of a new generalization of the {M}eixner polynomials.
\newblock {\em SIGMA Symmetry Integrability Geom. Methods Appl.}, 7:Paper 068,
11, 2011.

\bibitem {MR2944743}G.~Filipuk and W.~Van~Assche. \newblock Recurrence
coefficients of generalized {C}harlier polynomials and the fifth
{P}ainlev\'{e} equation. \newblock {\em Proc. Amer. Math. Soc.}. To appear.

\bibitem {MR1662690}M.~Foupouagnigni, M.~N. Hounkonnou, and A.~Ronveaux.
\newblock Laguerre-{F}reud equations for the recurrence coefficients of
{$D_{\omega}$} semi-classical orthogonal polynomials of class one.
\newblock In \emph{Proceedings of the {VIII}th {S}ymposium on {O}rthogonal
{P}olynomials and {T}heir {A}pplications ({S}eville, 1997)}, J. Comput. Appl.
Math. 99: 143--154, 1998.

\bibitem {MR1340932}A.~G. Garc{\'{\i}}a, F.~Marcell{\'a}n, and L.~Salto.
\newblock A distributional study of discrete classical orthogonal polynomials.
\newblock In \emph{Proceedings of the {F}ourth {I}nternational {S}ymposium on
{O}rthogonal {P}olynomials and their {A}pplications ({E}vian-{L}es-{B}ains,
1992)}, J. Comput. Appl. Math. 57: 147--162, 1995.

\bibitem {MR1467146}E.~Godoy, F.~Marcell{\'a}n, L.~Salto, and A.~Zarzo.
\newblock Perturbations of discrete semiclassical functionals by {D}irac
masses. \newblock {\em Integral Transforms Spec. Funct.}, 5(1-2):19--46, 1997.

\bibitem {MR1737084}M.~N. Hounkonnou, C. Hounga, and A.~Ronveaux.
\newblock Discrete semi-classical orthogonal polynomials: generalized
{C}harlier. \newblock {\em J. Comput. Appl. Math.}, 114(2):361--366, 2000.

\bibitem {MR2656096}R.~Koekoek, P.~A. Lesky, and R.~F. Swarttouw.
\newblock {\em Hypergeometric orthogonal polynomials and their
{$q$}-analogues}. \newblock Springer Monographs in Mathematics.
Springer-Verlag, Berlin, 2010.

\bibitem {MR1687531}A.~B.~J. Kuijlaars and W.~Van~Assche. \newblock Extremal
polynomials on discrete sets. \newblock {\em Proc. London Math. Soc. (3)},
79(1):191--221, 1999.

\bibitem {MR1665164}F.~Marcell{\'a}n and L.~Salto. \newblock Discrete
semi-classical orthogonal polynomials. \newblock {\em J. Differ. Equ. Appl.},
4(5):463--496, 1998.

\bibitem {MR2914891}F.~Marcell{\'{a}}n, M.~Sghaier, and M.~Zaatra. \newblock
On semiclassical linear functionals of class {$s=2$}. Classification and
integral representations. \newblock {\em J. Differ. Equ. Appl.}, 8:973--1000, 2012.

\bibitem {MR1270222}P.~Maroni. \newblock Une th\'eorie alg\'ebrique des
polyn\^omes orthogonaux. {A}pplication aux polyn\^omes orthogonaux
semi-classiques. \newblock In \emph{Orthogonal polynomials and their
applications ({E}rice, 1990)}, volume~9 of \emph{IMACS Ann. Comput. Appl.
Math.}, pages 95--130. Baltzer, Basel, 1991.

\bibitem {MR2457103}P.~Maroni and M.~Mejri. \newblock The symmetric
{$D_{\omega}$}-semi-classical orthogonal polynomials of class one.
\newblock {\em Numer. Algorithms}, 49(1-4):251--282, 2008.

\bibitem {MR1149380}A.~F. Nikiforov, S.~K. Suslov, and V.~B. Uvarov.
\newblock {\em Classical orthogonal polynomials of a discrete variable}.
\newblock Springer Series in Computational Physics. Springer-Verlag, Berlin, 1991.

\bibitem {MR2723248}F.~W.~J. Olver, D.~W. Lozier, R.~F. Boisvert, and C.~W.
Clark, editors. \newblock {\em N{IST} handbook of mathematical functions}.
\newblock U.S. Department of Commerce National Institute of Standards and
Technology, Washington, DC, 2010.

\bibitem {MR842801}A.~Ronveaux. \newblock Discrete semiclassical orthogonal
polynomials: generalized {M}eixner. \newblock {\em J. Approx. Theory},
46(4):403--407, 1986.

\bibitem {MR1780044}A.~Ronveaux. \newblock Asymptotics for recurrence
coefficients in the generalized {M}eixner case.
\newblock {\em J. Comput. Appl. Math.}, 133(1-2):695--696, 2001.

\bibitem {MR1939588}A.~Ronveaux and L.~Salto. \newblock Discrete orthogonal
polynomials---polynomial modification of a classical functional.
\newblock {\em J. Differ. Equ. Appl.}, 7(3):323--344, 2001.

\bibitem {MR2926308}C.~Smet and W.~Van~Assche. \newblock Orthogonal
polynomials on a bi-lattice. \newblock {\em Constr. Approx.} 36: 215--242, 2012.

\bibitem {MR2451211}W.~Van~Assche. \newblock Discrete {P}ainlev\'e equations
for recurrence coefficients of orthogonal polynomials. \newblock In
\emph{Difference equations, special functions and orthogonal polynomials},
pages 687--725. World Sci. Publ., Hackensack, NJ, 2007.

\bibitem {MR2063533}W.~Van~Assche and M.~Foupouagnigni. \newblock Analysis of
non-linear recurrence relations for the recurrence coefficients of generalized
{C}harlier polynomials. \newblock {\em J. Nonlinear Math. Phys.}, 10(suppl.
2):231--237, 2003.
\end{thebibliography}
\end{document}